\documentclass{amsart}
\usepackage{amsfonts,amssymb,amsmath,amsthm}
\usepackage{pb-diagram}
\usepackage{url}
\usepackage{enumerate}

\newtheorem{thrm}{Theorem}[section]
\newtheorem{lem}[thrm]{Lemma}
\newtheorem{prop}[thrm]{Proposition}

\newtheorem{conj}[thrm]{Conjecture}
\theoremstyle{definition}

\newtheorem{rem}[thrm]{Remark}

\numberwithin{equation}{section}

\author{F. Pazuki}
\address{
IMB, Univ. Bordeaux 1\\
351, cours de la Lib\'eration\\
33405 Talence, France\\}
\email{fabien.pazuki@math.u-bordeaux1.fr}
\thanks{Many thanks to P. Autissier, S. David and G. R\'emond for their questions and remarks, and to the anonymous referee for further remarks. Thanks to the University of Bordeaux 1 for supporting the conference ``Jeudynamiques'' in 2010.}


\begin{document}

\title[Polarized morphisms between abelian varieties]{Polarized morphisms between abelian varieties}

\begin{abstract}
In this paper we will study the Dynamical Manin-Mumford problem, focusing on the question of polarizability for endomorphisms of an abelian variety $A$ and on the action of a Frobenius and its Verschiebung on the diagonal subvariety of $A\times A$. We complete the study with different polarizability criteria.
\end{abstract}
\maketitle

\section{Introduction}
  
   In this work we will study the role of polarizability for morphisms between abelian varieties in the Dynamical Manin-Mumford problem. The results also provide a way to strengthen the article \cite{Paz}. 
   
   Let us recall a few definitions: an endomorphism $\psi:X\to X$ of a projective variety is said to have a \textit{polarization} if there exists an ample divisor $D$ on $X$ such that $\psi^{*}D\sim d D$ for some $d>1$, where $\sim$ stands for linear equivalence. Another (equivalent) way of defining this notion is to use an ample line bundle $\mathcal{L}$ over $X$ such that $\psi^{*}\mathcal{L}= \mathcal{L}^{\otimes d}$. The integer $d$ is called the \textit{weight} of the morphism $\psi$.
   
A subvariety $Y$ of $X$ is \textit{preperiodic} under $\psi$ if there exists integers $m\geq0$ and $s>0$ such that $\psi^{m+s}(Y)=\psi^{m}(Y)$. We denote $\mathrm{Prep}_{\psi}(X)$ as the set of preperiodic points of $X$ under $\psi$. We will focus on the dynamical Manin-Mumford Conjecture 1.2.1 in \cite{Zha}:

\begin{conj}\label{Zhang}
(Algebraic Dynamical Manin-Mumford) Let $\psi\!\!: X\to X$ be an endomorphism of a projective variety over a number field $k$ with a polarization, and let $Y$ be a subvariety of $X$. If $Y\cap \mathrm{Prep}_{\psi}(X)$ is Zariski dense in $Y$, then $Y$ is a preperiodic subvariety.
\end{conj}

This conjecture is very natural as it contains the classical Manin-Mumford conjecture proved by Raynaud in 1983: if one chooses $\psi=[n]$ and $X=A$ an abelian variety, then $\psi$ is polarized with weight $n^2$ and the preperiodic points are just torsion points.

One can find some examples of non-trivial endomorphisms where Conjecture \ref{Zhang} is true for the diagonal subvariety of a power of an abelian variety. Let $A$ be an abelian variety and $\mathcal{L}$ an ample symmetric line bundle. Consider the endomorphisms $ \alpha$ and $ \beta$ on $A^4$ given by $ \alpha(x,y,z,t)=(x+z,y+t,x-z,y-t)$ and $\beta(x,y,z,t)=(x+y,x-y,z+t,z-t)$. If we denote $\mathcal{L}_4=p_1^*\mathcal{L}\otimes p_2^*\mathcal{L}\otimes p_3^*\mathcal{L} \otimes p_4^*\mathcal{L}$, then $\alpha^*\mathcal{L}_4=\beta^*\mathcal{L}_4=\mathcal{L}_4^{\otimes 2}$, thus the morphism $(\alpha,\beta)$ is polarized by $\mathcal{L}_4$. Of course one has $\alpha^2=\beta^2$, hence the diagonal of $A^4\times A^4$ is actually preperiodic under the morphism $(\alpha,\beta)$. 

One can consult \cite{BaHsi} for other examples where this conjecture is proved, namely on $\mathbb{P}^1\times\mathbb{P}^1$ under a coordinatewise polynomial action.

D. Ghioca and T. Tucker found a family of counterexamples to this conjecture in \cite{GhTuZh}. They use squares of elliptic curves with complex multiplication. This was followed by the paper \cite{Paz} where counterexamples of dimension four were constructed over squares of CM abelian surfaces using complex multiplication on curves of genus 2 and some polarizability lemma. In this work we show a more general theorem, valid in greater dimension as well. The main result obtained is Theorem \ref{Frob}, which gives a general way of producing dynamical systems with particular intersection properties. 

 Let $k$ be a CM number field and $\bar{k}$ an algebraic closure. We will also denote the complex conjugation by a bar, but the context will always be clear. Let $\mathcal{O}_k$ be the ring of integers of $k$. Let $A$ be an abelian variety over $k$ with complex multiplication by $k$. Let $S_1$ be the set of ramified places of $k$ and let $S_2$ be the set of places where $A$ has bad reduction. Let $S_3$ be the set of places $\frak{a}$ such that $F_{\frak{a}}^n= V_{\frak{a}}^n$ for a positive integer $n$, where $F_\frak{a}$ is the endomorphism of $A$ which lifts the geometric Frobenius and $V_{\frak{a}}$ the associated Verschiebung (this set is small, see remark \ref{CM unit}). We will denote by $S=S_1\cup S_2\cup S_3$ the union of these places.

\begin{thrm}\label{Frob}
Let $A$ be an abelian variety defined over a number field $k$ that contains the field of complex multiplications $\mathrm{End}_{\bar{k}}(A)\otimes{\mathbb{Q}}$ and such that $\mathcal{O}_k\subset{\mathrm{End}_{\bar{k}}(A)}$.  Let $\frak{p}\notin{S}$. Let $F_{\frak p}$ denote the Frobenius associated to $\frak{p}$ and $V_{\frak p}$ denote the Verschiebung associated to $F_{\frak p}$. Then $F_{\frak p}$ and $V_{\frak p}$ are polarizable and the dynamical system $(A\times A, F_{\frak p}\times V_{\frak p})$, together with the diagonal subvariety of $A\times A$, gives a counterexample to Conjecture \ref{Zhang}.
\end{thrm}

This statement shows that one can actually use any CM abelian variety to construct a counterexample. Moreover, the manner in which these examples are produced is more geometric and explains the origin of the first examples. 
 
There is a new version of Conjecture \ref{Zhang} that can be found in \cite{GhTuZh},  where one takes into account the action on the tangent space of $T_{X,x}$ at preperiodic points $x$. More precisely one has:

\begin{conj}\label{GhiocaTuckerZhang} (Ghioca, Tucker, Zhang)
Let $\psi\!\!: X\to X$ be an endomorphism of a smooth projective variety over a number field $k$ with a polarization, and let $Y$ be a subvariety of $X$. Then $Y$ is preperiodic under $\psi$ if and only if there exists a Zariski dense subset of points $x\in{Y\cap\mathrm{Prep}_{\psi}(X)}$ such that the tangent subspace of $Y$ at $x$ is preperiodic under the induced action of $\psi$ on the Grassmanian $\mathrm{Gr}_{\mathrm{dim}(Y)}(T_{X,x})$.
\end{conj}

\begin{rem}Theorem 2.1 of \cite{GhTuZh} implies that Conjecture \ref{GhiocaTuckerZhang} holds when $\psi$ is a group endomorphism. Hence our constructions provide examples of subvarieties $Y$ containing a Zariski dense set of preperiodic points $x$ but such that $T_{Y,x}$ is not preperiodic under the action of $\psi$. Another conjecture concerning this problem can be found in a recent work of Yuan and Zhang \cite{YuZh}.
\end{rem}

Using Latt\`es maps one can transport these constructions on $\mathbb{P}^1\times\mathbb{P}^1$, hence using the Segre embedding, also on $\mathbb{P}^3$. A good question for future work would then be: what happens on $\mathbb{P}^2$? In the next section we give the proof of Theorem \ref{Frob}. Then, as previously stated in \cite{Paz}, \cite{DaHi} and \cite{Pol}, obtaining polarizability can be difficult and is linked to arithmetic properties of the field $k$. Thus, we complete this work by giving some explicit polarizability criteria for morphisms between abelian varieties.

\section{Abelian varieties and Frobenius maps}

Let us start with the following lemma:

\begin{lem}\label{pola compo}
Let $A$ be a projective variety over a field $k$, of dimension $g$ and let $\mathcal{L}$ be an ample line bundle. Let $f,j,h$ be three endomorphisms of $A$ such that $f\circ j=h$. Suppose $f^*\mathcal{L}=\mathcal{L}^{\otimes d}$ and $h^*\mathcal{L}=\mathcal{L}^{\otimes n}$. Then $j$ is polarized by $\mathcal{M}=\mathcal{L}^{\otimes d}$, $d\mid n$ and $j^*\mathcal{M}=\mathcal{M}^{\otimes \frac{n}{d}}$.
\end{lem}
\begin{proof}
We calculate $(f\circ j)^*\mathcal{L}=j^*(f^*\mathcal{L})=j^*(\mathcal{L}^{\otimes d})$, hence $j^*(\mathcal{L}^{\otimes d})=\mathcal{L}^{\otimes n}$. Take the degree (associated to $\mathcal{L}$) to obtain: $\deg_{\mathcal{L}}(j)<\mathcal{L}^{\otimes d}>^g=<\mathcal{L}^{\otimes n}>^g$, where $<.>^g$ stands for the self-intersection product. It shows that $\deg_\mathcal{L}(j) d^g <\mathcal{L}>^g=n^g <\mathcal{L}>^g$, which is equivalent to $\deg_\mathcal{L}(j) d^g=n^g$ as $\mathcal{L}$ is ample. So there exists an integer $m$ such that $n=dm$. This shows $j^*(\mathcal{L}^{\otimes d})=(\mathcal{L}^{\otimes d})^{\otimes m}$.
\end{proof}

We will now move on to the proof of Theorem \ref{Frob}. 

\begin{proof}
 One remark is that if $k$ is large enough for $A$ to obtain semi-stable reduction, then $S_2$ is empty. Indeed, as $A$ has complex multiplications, it has potentially good reduction everywhere. One finds in \cite{DaHi}, on  pages 16-17, Proposition 3.2-3.3, that if $\frak{p}\notin{S_1\cup S_2}$, then the Frobenius $F_{\frak p}$ is polarized by a symmetric line bundle $\mathcal{L}$. 

Let $V_{\frak p}$ denote the Verschiebung associated to $F_{\frak p}$. Let $N=\mathrm{Norm}({\frak p})$. We know that $F_{\frak p}^*(\mathcal{L})=\mathcal{L}^{\otimes N}$ and that $[N]^*{\mathcal{L}}=\mathcal{L}^{N^2}$. As we have $F_{\frak p}\circ V_{\frak p}=[N]$, we can apply Proposition \ref{pola compo} to obtain $V_{\frak p}^*(\mathcal{L}^{\otimes N})=(\mathcal{L}^{\otimes N})^{\otimes N}=\mathcal{L}^{\otimes N^2}$, hence $V_{\frak p}$ is also polarizable, and with the same weight. Thus $F_{\frak p}\times V_{\frak p}$ is also polarized on $A\times A$.

Now consider $\Delta=\{(P,P)\,\vert\, P\in{A}\}$, the diagonal subvariety of $A\times A$. This variety cannot be preperiodic under $\varphi=F_{\frak p}\times V_{\frak p}$ because it would imply $F_{\frak p}^{m}=V_{\frak p}^m$ for a positive integer $m$, which is impossible because it would imply $\frak{p}\in{S_3}$. However, $\Delta\cap\mathrm{Prep}_{\varphi}(A\times A)$ is dense in $\Delta$ because it contains all torsion points of $\Delta$.

\end{proof}

\begin{rem}
This theorem is a way to generalize the previous counterexamples of \cite{GhTuZh} and \cite{Paz}. For example, when an elliptic curve $E$ has a multiplication by $[i]$, the morphism $[2+i]$ corresponds to the Frobenius $F_{(2-i)}$, as is shown in \cite{Sil}, Proposition 4.2, page 122.
\end{rem}

\begin{rem}\label{popo}
This theorem explains as a by-product that one needs to avoid the ramified places if one searches for polarizability in the CM case. The discriminant of the number field $\mathbb{Q}[i]$  is $-4$. As $(1+i)\,\vert\, (2)$, the morphism $[1+i]$ will not be polarizable. On the contrary, as $(2+i)$ and $(2)$ are coprime ideals, $[2+i]$ is polarizable. This refines the results in \cite{Pol}.
\end{rem}

\begin{rem}
One can construct examples of polarized morphisms on $A\times A$ like in Theorem \ref{Frob} as soon as $\mathbb{Z}\subsetneq\mathrm{End}(A)$ and the Rosati involution is not trivial. We refer to \cite{Mum}, page 200, Theorem 2, for the classification of division algebras that can occur for $\mathrm{End}_\mathbb{Q}(A)$.
\end{rem}

\section{Polarizability criteria}

A classical tool to obtain polarizability is the cube theorem. We refer to \cite{Paz} for some formulas derived from this theorem and useful to obtain information on the weight of complex multiplications. 

In this section we give other polarizability criteria, namely Proposition \ref{Rosati} for the action of the Rosati involution and Proposition \ref{elliptic} for the particular case of elliptic curves.

\subsection{Rosati involution}

Let $A$ be an abelian variety over a field $k$ and $\mathcal{L}$ be an ample line bundle. We denote by $\dagger$ the Rosati involution associated to $\mathcal{L}$. (See \cite{Mil}, page 137, for more details.) For any invertible line bundle $\mathcal{M}$, we let $\varphi_{\mathcal{M}}$ denote the classical map $a\mapsto t^*_a\mathcal{M}\otimes \mathcal{M}^{-1}$ from $A$ to $\mathrm{Pic}^0(A)$.

\begin{prop}\label{Rosati}
Let $A$ be an abelian variety, $\psi$ an endomorphism of $A$ and $\mathcal{L}$ an ample line bundle. Then $\psi$ is polarized by $\mathcal{L}$ with weight $d$ if and only if one has $\psi^\dagger \psi =[d]$, where $\dagger$ is associated to $\mathcal{L}$.
\end{prop}

\begin{proof}
Let $\mathrm{NS}_{\mathbb{Q}}(A)$ be the N\'eron-Severi group of $A$ tensored by $\mathbb{Q}$. Let $\mathrm{End}^{\dagger}_{\mathbb{Q}}(A)$ be the group of endomorphism fixed by $\dagger$.  We thus have an isomorphism (see \cite{Mil}, page 137)

\begin{align*} 
 H\!\!: \mathrm{NS}_{\mathbb{Q}}(A) &\longrightarrow \mathrm{End}^{\dagger}_{\mathbb{Q}}(A)\\
  \mathcal{M}& \mapsto \varphi_{\mathcal{L}}^{-1}\circ \varphi_{\mathcal{M}}. 
\end{align*}

We then calculate that $H(\psi^*\mathcal{L})=\psi^{\dagger} \psi $ and $H(\mathcal{L}^{\otimes d})=[d]$, thus we can express the polarizability condition $\psi^*\mathcal{L}=\mathcal{L}^{\otimes d}$ by $\psi^{\dagger} \psi =[d]$. 
\end{proof}

\begin{rem}\label{CM unit}
If $A$ has complex multiplication, then one has $\psi^{\dagger}=\overline{\psi}$. In order to study Conjecture \ref{Zhang}, we need to find endomorphisms $\psi$ such that for all integers $m\geq1$, $\psi^m\neq (\psi^{\dagger})^m$. Hence, in the CM case, one needs to find a number $\alpha$ such that $\alpha \overline{\alpha}\in{\mathbb{Z}}$, $\alpha\notin{\mathbb{Z}}$ and $\alpha/\overline{\alpha}$ is not a root of unity. If $\alpha=a+ib$ where $a,b\in{\mathbb{Z}}$ and $b\neq0$, then $\alpha/\overline{\alpha}$ has norm one and there exists a real number $\theta$ such that $\sin\theta=2ab/(a^2+b^2)$. One just needs to check that $\theta$ is not always a rational number. It is very rarely the case, as $\sin\theta$ is itself rational.
\end{rem}

\textbf{Application:} \textit{Multiplication by $1+\zeta_{5}$ not polarized by $\Theta$.} Thanks to this Rosati action, we can add a remark to one of the examples of \cite{Paz} where the morphism $[1+\zeta_5]$ is not polarized by the divisor $\Theta$ on the jacobian of a genus 2 curve. A little calculation in the particular example of $z=4+3\zeta_5+12\zeta_5^2$ gives $z\overline{z}=121$ and $z/\overline{z}$ is not a root of unity. Hence, by Proposition \ref{Rosati}, we obtain that $[z]$ is polarized by $\Theta$.

\subsection{Elliptic curves}

In the particular case of elliptic curves, one can obtain a precise condition for polarizability. There are two reasons for this:  the variety is simple and the divisor support is just a point.
\begin{prop}\label{elliptic}
Let $E$ be an elliptic curve over a field $k$ of characteristic zero and $f$ be an isogeny of $E$ of degree $d$. We denote by $E[2]$ the group of $2$-torsion points over $\bar{k}$. Then we have $$\Big(\mathrm{Card}(E[2]\cap \mathrm{Ker}(f))\neq 2 \Big)\Leftrightarrow \Big(f^*(O)\sim d(O)\Big).$$
\end{prop}

\begin{proof}
Let $H=E[2]\cap \mathrm{Ker}(f)$ and $G=\mathrm{Ker}(f)\backslash H$. Let $m=\mathrm{Card}(H)$. We know that $m\in\{1,2,4\}$. Let us choose any short Weierstrass model $y^2=x^3+ax+b$. We calculate:  $$f^*(O)=\sum_{P\in{\mathrm{Ker}(f)}}(P)=\sum_{P\in{H}}(P)+\sum_{P\in{G}}(P)=\sum_{P\in{H}}(P)+\sum_{P\in{G/\pm1}}((P)+(-P)),$$
and as $\mathrm{div}(x-x(P))=(P)+(-P)-2(O)$, we have $$\displaystyle{\sum_{P\in{G/\pm1}}((P)+(-P))\sim 2\Big(\frac{d-m}{2}\Big)(O)},$$ thus $\displaystyle{f^*(O)\sim \sum_{P\in{H}}(P)+(d-m)(O)}$. Then we have three cases: 
\begin{itemize}
\item[$\bullet$] if $m=1$, then $H=\{O\}$ and $f^{*}(O)\sim (O)+(d-1)(O) \sim d(O)$,
\item[$\bullet$] if $m=2$, then $H=\{O,\,P\}$ and $f^{*}(O)\sim (O)+(P)+(d-2)(O) \sim (P)+(d-1)(O)$, then the divisor $D=(P)-(O)$ is of degree $0$ but is not the divisor of a rational function on $E$ because $P-O\neq O$,
\item[$\bullet$] if $m=4$, then $H=\{O,P_1,P_2, P_1+P_2\}$ and $f^{*}(O)\sim (O)+(P_1)+(P_2)+(P_1+P_2)+(d-4)(O) \sim d(O)$ because $\mathrm{div}(y)=(P_1)+(P_2)+(P_1+P_2)-3(O)$.
\end{itemize}
\end{proof}

\begin{rem}\label{rem elliptic}
Let $E$ be the elliptic curve with affine model $y^2=x^3+x$. It has complex multiplication by $[i]:(x,y)\to (-x,iy)$. Then the isogeny $[1+i]$ has degree $2$, and as $[2]=[1+i][1-i]$, we obtain $\mathrm{Ker}[1+i]\subset E[2]$, hence $\mathrm{Card}(E[2]\cap \mathrm{Ker}[1+i])=2$ and by Proposition \ref{elliptic} the morphism $[1+i]$ is not polarized. As shown in \cite{Paz}, the morphism $[2+i]$ is polarized. Another proof of this fact is that $[5]=[2+i][2-i]$, hence $\mathrm{Ker}[2+i]\subset E[5]$, but $E[5]\cap E[2]=O$ and again by Proposition \ref{elliptic} the morphism $[2+i]$ is polarized, with weight $5$. 
\end{rem}

\begin{rem}
This result refines the remark made in the introduction of \cite{Pol} concerning elliptic curves. See the remark \ref{popo} for more details.
\end{rem}

\subsection{Explicit Latt\`es dynamical system}

Let $E$ be the elliptic curve with affine model $y^2=x^3+x$. It has complex multiplication by $[i]\!\!:(x,y)\to (-x,iy)$. Let $\pi\!\!:E\to \mathbb{P}^1$ be defined by $\pi(x,y)=x$ and $\pi(O)=\infty$. Then we obtain the following picture:

\dgARROWLENGTH=0.5cm %(longueur de la flèche)

\[
\begin{diagram}
\node{E}
\arrow[3]{e,t}{[2+i]}
\arrow[3]{s,l}{\pi}
\node[3]{E}
\arrow[3]{s,r}{\pi}\\\\\\
\node{\mathbb{P}^1}
\arrow[3]{e,b}{\varphi}
\node[3]{\mathbb{P}^1}
\end{diagram}
\]

A direct calculation shows that the Latt\`es map $\varphi$ can be expressed on the affine chart as $$\varphi(x)=\frac{(3-4i)x(x^2+1-2i)^2}{(5x^2+1+2i)^2}.$$ 

\begin{rem}
We thus obtain the $x$-coordinate of the four non-trivial $[2+i]$-torsion points: $\pm x=\sqrt{\frac{\sqrt{5}-1}{10}}+i\,\frac{1}{5}\sqrt{\frac{10}{\sqrt{5}-1}}$.
\end{rem}

For the Latt\`es map of $[2-i]$, we obtain on the affine chart $$\psi(x)=\frac{(3+4i)x(x^2+1+2i)^2}{(5x^2+1-2i)^2}.$$

Now consider
\begin{align*} 
 \delta\!\!:\mathbb{P}^1\times \mathbb{P}^1 &\longrightarrow \mathbb{P}^1\times \mathbb{P}^1\\
 (s,t)& \longmapsto (\varphi(s),\psi(t)). 
\end{align*}

Then if $D=\{\infty\}\times\mathbb{P}^1+\mathbb{P}^1\times\{\infty\}$, we obtain $\delta^*D\sim5D$, hence $\delta$ is polarized by $D$ with weight $5$. This gives an explicit counterexample to Conjecture \ref{Zhang} in the case of $X=\mathbb{P}^1\times \mathbb{P}^1$ as in the constructions of \cite{GhTuZh}. See \cite{Fak} for Latt\`es dynamical systems.

\end{document}